\documentclass{article}%
\usepackage{amssymb}
\usepackage{amsmath}
\usepackage{amsfonts}
\usepackage{graphicx}%
\providecommand{\U}[1]{\protect\rule{.1in}{.1in}}
\newtheorem{theorem}{Theorem}

\newenvironment{proof}[1][Proof]{\noindent\textbf{#1.} }{\ \rule{0.5em}{0.5em}}
\begin{document}

\title{\textbf{SPECIAL SPACE CURVES CHARACTERIZED BY }$\det(\alpha^{(3)},\alpha
^{(4)},\alpha^{(5)})=0$}
\author{\textbf{Yusuf YAYLI}\thanks{Ankara University, Faculty of Science, Department
of Mathematics, Ankara, TURKEY}\textbf{, Semra SARACOGLU}\thanks{Siirt
University, Faculty of Science and Arts, Department of Mathematics, Siirt,
TURKEY}}
\maketitle

\begin{abstract}
In this study, by using the facts that $\det(\alpha^{(1)},\alpha^{(2)}%
,\alpha^{(3)})$ $=$ $0$ characterizes plane curve, and $\det(\alpha
^{(2)},\alpha^{(3)},\alpha^{(4)})$ $=$ $0$ does a curve of constant slope, we
give the special space curves that are characterized by $\det(\alpha
^{(3)},\alpha^{(4)},\alpha^{(5)})$ $=$ $0,$ in different approaches. We find
that the space curve is Salkowski if and only if $\det(\alpha^{(3)}%
,\alpha^{(4)},\alpha^{(5)})$ $=$ $0.$ The approach we used in this paper is
useful in understanding the role of the curves that are characterized by
$\det(\alpha^{(3)},\alpha^{(4)},\alpha^{(5)})=0$ in differential geometry.

\textbf{AMS Subj. Class.:} 53A04, 53A05, 53B30 .


\end{abstract}

\section{INTRODUCTION}

In the classical differential geometry general helices, Salkowski curves and
spherical helices are well-known curves. These curves are defined by the
property that the tangent makes a constant angle with a fixed straight line
(the axis of the general helix) [2-5]. Among all the space curves with
constant curvature, Salkowski curves are those for which the normal vector
maintains a constant angle with a fixed direction in the space [3]. The study
of these curves is given by Salkowski in [4], Monterde in [3], Kula,
Ekmek\c{c}i, Yayl\i, \.{I}larslan in [1] and Takenaka in [6]. Throughout this
paper, by using some characterizations from [6], we are going to present the
characterizations of the condition that $\det(\alpha^{(3)},\alpha^{(4)}%
,\alpha^{(5)})=0$ for special space curves in different approaches.

We first introduce the space curve $\alpha$ as
\[
\alpha(s)=(x(s),y(s),z(s)),
\]
where $s$ is the arclength parameter and denote two geometrical quantities,
curvature and torsion, by $\kappa$ and $\tau.$ These structures play essential
role in the theory of space curve. Such as, circles and circular helices are
curves with constant curvature and torsion [3].

Furthermore, as it is seen from the known facts in [2,5,6],

(1) The condition that $\det(\alpha^{(0)},\alpha^{(1)},\alpha^{(2)})=0$
characterizes a great circle, where $\alpha$ is a spherical curve,

(2) The condition that $\det(\alpha^{(1)},\alpha^{(2)},\alpha^{(3)})=0$
characterizes a plane curve,

(3) The condition that $\det(\alpha^{(2)},\alpha^{(3)},\alpha^{(4)})=0$
characterizes a curve of constant slope.

In [6], Takenaka represents the diffuculty in solving $\det(\alpha
^{(3)},\alpha^{(4)},\alpha^{(5)})=0.$ Therefore, he has given the following
form to put the complicated terms away%
\begin{equation}
\det(\alpha^{(3)},\alpha^{(4)},\alpha^{(5)})=\kappa^{4}\det\left(
\begin{array}
[c]{ccc}%
\varphi_{1} & \varphi_{2} & \varphi_{3}\\
\kappa & \kappa^{\prime} & \kappa^{\prime\prime}\\
\tau & \tau^{\prime} & \tau^{\prime\prime}%
\end{array}
\right)
\end{equation}
where
\begin{align}
\varphi_{1} &  =-(\frac{1}{\kappa})^{\prime}\\
\varphi_{2} &  =-(\frac{1}{\kappa})^{\prime\prime}-\frac{1}{\kappa}(\kappa
^{2}+\tau^{2}),\nonumber\\
\varphi_{3} &  =-(\frac{1}{\kappa})^{\prime\prime\prime}-\left\{  \frac
{1}{\kappa}(\kappa^{2}+\tau^{2})\right\}  ^{\prime}-\frac{1}{2\kappa}\left(
\kappa^{2}+\tau^{2}\right)  ^{\prime}.\nonumber
\end{align}

Moreover, by taking a curve with constant curvature, he shows that the
following conditions are equivalent:

$i.)$ $\det(\alpha^{(2)},\alpha^{(3)},\alpha^{(4)})=0$

$ii.)$ $\tau(s)=\mp\dfrac{a^{3}(bs+c)}{\left\{  1-a^{4}(bs+c)^{2}\right\}
^{1/2}}$, where $b,c\in%
\mathbb{R}
$ and
\[
-\frac{1}{b}(1+c)\leq s\leq\frac{1}{b}(1-c).
\]
Then for the constant curvature $\kappa=a,$ he has found the torsion as:%
\begin{equation}
\tau(s)=\pm\frac{a^{3}(bs+c)}{\left[  1-a^{4}(bs+c)^{2}\right]  ^{1/2}}.
\end{equation}
On the other hand, in [3], Monterde shows that normal vectors of the curve
$\alpha$ parametrized by arc-length with $\kappa=1$ make a constant angle with
a fixed line in space if and only if
\begin{equation}
\tau(s)=\pm\dfrac{s}{\left[  \tan^{2}\phi-s^{2}\right]  ^{1/2}}.
\end{equation}

As it is seen Salkowski curves are slant helices. In this study, we
characterize Salkowski curves by the help of determinants. Accordingly, we
present new special characterization with slant helices.

Analogously, in this paper, considering the relationship between the space
curves, we have found important results. Moreover, we obtain that with
constant curvature $\kappa=1,$ the space curve is Salkowski if and only if
\begin{equation}
\det(\alpha^{(3)},\alpha^{(4)},\alpha^{(5)})=0
\end{equation}
and also we have showed the following three conditions are equivalent:

(1) The space curve $\alpha$ is Salkowski curve,

(2) $\det(\alpha^{(3)},\alpha^{(4)},\alpha^{(5)})=0$,

(3) $\tau(s)=\pm\dfrac{s}{\left[  \tan^{2}\phi-s^{2}\right]  ^{1/2}}.$

If we look previous studies in that field, we only meet the studies on the
calculations of curvatures. We have never seen a study about the calculations
on the family of determinants. In this study, we try to give the special space
curves that are characterized by $\det(\alpha^{(3)},\alpha^{(4)},\alpha
^{(5)})$ $=$ $0,$ in different approaches. We hope that this study will gain
different interpretation to the other studies in this field.

\section{SPHERICAL INDICATRICE CURVES}

In this section, we give the spherical indicatrice curves that are
characterized by $\det(\alpha^{(3)},\alpha^{(4)},\alpha^{(5)})=0$ and discuss
the main properties.

\begin{theorem}
Let
\begin{equation}
\alpha:I\rightarrow E^{3}%
\end{equation}%
\[
\text{ \ \ \ \ \ \ \ }s\mapsto\alpha(s)
\]
be a space curve that is parametrized arc-length with $\kappa\equiv1.$ The
tangent indicatrice of the space curve $\left(  T\right)  $ is spherical helix
if and only if
\begin{equation}
\det(\alpha^{(3)},\alpha^{(4)},\alpha^{(5)})=0.
\end{equation}

\end{theorem}

\begin{proof}
Let $\left(  T\right)  $ be tangent curve as:%
\begin{equation}
\alpha^{\prime}(s)=(T)
\end{equation}
and then%
\begin{equation}
T^{\prime}=\kappa N=N
\end{equation}
for $\kappa\equiv1.$ Here, $\left\Vert T^{\prime}(s)\right\Vert =1$ and the
arc-parameter of the curve $(T)$ is $s$ and also the parameters of the curves
$\alpha$ and $\left(  T\right)  $ are the same. According to all of these, if
$(T)$ is helix, then%
\begin{equation}
\det(T^{(2)},T^{(3)},T^{(4)})=0,\text{ }[6]
\end{equation}
Thus%
\begin{equation}
T=\alpha^{\prime}(s)
\end{equation}
and then
\begin{equation}
\det(\alpha^{(3)},\alpha^{(4)},\alpha^{(5)})=0.
\end{equation}
On the contrary, it can be easily proved as above smiliarly.
\end{proof}

\begin{theorem}
Let
\[
\alpha:I\rightarrow E^{3}%
\]%
\[
\text{ \ \ \ \ \ \ \ }s\mapsto\alpha(s)
\]
be a space curve that is parametrized arc-length with $\kappa\equiv1.$ The
tangent indicatrice $\left(  T\right)  $ of $\alpha$ is spherical helix if and
only if
\begin{equation}
2\tau^{\prime\prime}(1+\tau^{2})-3\tau(1+\tau^{2})^{\prime}=0
\end{equation}

\end{theorem}

\begin{proof}
The tangent indicatrix of the space curve $\alpha$ is spherical helix if and
only if
\[
\det(\alpha^{(3)},\alpha^{(4)},\alpha^{(5)})=0
\]
and also from [6]%

\begin{equation}
\det(\alpha^{(3)},\alpha^{(4)},\alpha^{(5)})=\kappa^{4}\det\left[
\begin{array}
[c]{ccc}%
\varphi_{1} & \varphi_{2} & \varphi_{3}\\
\kappa & \kappa^{\prime} & \kappa^{\prime\prime}\\
\tau & \tau^{\prime} & \tau^{\prime\prime}%
\end{array}
\right]  .
\end{equation}
Here, for $\kappa\equiv1,$%
\[
\det(\alpha^{(3)},\alpha^{(4)},\alpha^{(5)})=\det\left[
\begin{array}
[c]{ccc}%
0 & -(1+\tau^{2}) & -(1+\tau^{2})^{\prime}-\frac{1}{2}(1+\tau^{2})^{\prime}\\
1 & 0 & 0\\
\tau & \tau^{\prime} & \tau^{\prime\prime}%
\end{array}
\right]  =0
\]
then we obtain%
\begin{equation}
2\tau^{\prime\prime}(1+\tau^{2})-3\tau(1+\tau^{2})^{\prime}=0
\end{equation}
Here, $\tau$ is torsion of the curve $\alpha(s).$
\end{proof}

\begin{theorem}
Let
\[
\alpha:I\rightarrow E^{3}%
\]%
\[
\text{ \ \ \ \ \ \ \ }s\mapsto\alpha(s)
\]
be a space curve that is parametrized arc-length with $\kappa\equiv1.$ The
tangent indicatrice $(T)$ of the space curve $\alpha$ is spherical helix if
and only if
\begin{equation}
\tau(s)=\pm\frac{bs+c}{\left[  1-(bs+c)^{2}\right]  ^{1/2}}%
\end{equation}
where $b\neq0$ and $b,c\in%
\mathbb{R}
$,%
\[
-\frac{1}{b}(1+c)\leq s\leq\frac{1}{b}(1-c).
\]

\end{theorem}

\begin{proof}
Under the condition $\kappa\equiv1,$ we have $\varphi_{1}=0,$ $\varphi
_{2}=-(1+\tau^{2})$ and $\varphi_{3}=-(1+\tau^{2})^{\prime}-\frac{1}{2}%
(1+\tau^{2})^{\prime}$ from the proof of the theorem above . By giving
similliar calculations from [6], the theorem can be easily proved.
\end{proof}

\textbf{Result. }By making similiar calculations we can get that
\begin{equation}
\tau=\pm\frac{bs+c}{\left[  1-(bs+c)^{2}\right]  ^{1/2}}%
\end{equation}
where $b=\dfrac{1}{\tan\phi},$ is the solution of differential equation as
given:%
\begin{equation}
\det(\alpha^{(3)},\alpha^{(4)},\alpha^{(5)})=2\tau^{\prime\prime}(1+\tau
^{2})-3\tau(1+\tau^{2})^{\prime}=0
\end{equation}

On the other hand, the following figure can be given as an example for showing
the curve whose $\kappa$ and $\tau$ satisfy the condition above for
$\kappa=1:$%
\[%
{\includegraphics[
natheight=2.452600in,
natwidth=7.141600in,
height=1.4875in,
width=4.3016in
]%
{LYB2GX01.wmf}%
}%
\]

\[
Figure1.\text{Salkowski cuves with }\kappa\equiv1\text{ [3]}%
\]

\section{CHARACTERIZATIONS OF SALKOWSKI CURVES}

In this section, we will introduce the condition that $\det(\alpha
^{(3)},\alpha^{(4)},\alpha^{(5)})=0$ for Salkowski curves in different
approaches .

\begin{theorem}
Let
\[
\alpha:I\rightarrow E^{3}%
\]%
\[
\text{ \ \ \ \ \ \ \ }s\mapsto\alpha(s)
\]
be a space curve that is parametrized arc-length with $\kappa\equiv1.$ The
space curve $\alpha$ is Salkowski curve if and only if
\[
\det(\alpha^{(3)},\alpha^{(4)},\alpha^{(5)})=0.
\]
\ 
\end{theorem}

\begin{proof}
If the space curve $\alpha$ is Salkowski and slant helix, then
\begin{equation}
\left\langle N,d\right\rangle =\cos\theta=\text{ constant},
\end{equation}
for a fixed line $d$ in a space. Besides, it can be easily given that for the
normal vector $N$ of the space curve $\alpha:$%
\begin{equation}
\beta(s)=T\text{ and }\beta^{\prime}(s)=N\text{ for }\kappa\equiv1.
\end{equation}
Here,
\begin{equation}
\left\langle \beta^{\prime}(s),d\right\rangle =\left\langle N,d\right\rangle
=cons\tan t.
\end{equation}
In this case, the curve $\beta(s)$ is spherical helix. From the Theorem.1, for
$\beta=\alpha^{\prime},$%
\begin{equation}
\det(\beta^{(2)},\beta^{(3)},\beta^{(4)})=0\text{ }%
\end{equation}
and then
\[
\det(\alpha^{(3)},\alpha^{(4)},\alpha^{(5)})=0.
\]
In contrary,
\[
\det(\alpha^{(3)},\alpha^{(4)},\alpha^{(5)})=0
\]
then for $\beta=\alpha^{(1)}$, we obtain
\[
\det(\beta^{(2)},\beta^{(3)},\beta^{(4)})=0\text{.}%
\]
Thus, $\beta(s)$ is a spherical helix. In that case, the curve $\alpha(s)$ is
slant helix. Then, we can easily get that $\alpha(s)$ is Salkowski for
$\kappa\equiv1$ from the calculations above.

\begin{theorem}
Let
\begin{equation}
\alpha:I\rightarrow E^{3}%
\end{equation}%
\[
\text{ \ \ \ \ \ \ \ }s\mapsto\alpha(s)
\]
be a space curve that is parametrized arc-length with $\kappa\equiv1.$ The
following three conditions are equivalent:

\textbf{1.) }The space curve $\alpha$ is Salkowski curve,

\textbf{2.) }$\det(\alpha^{(3)},\alpha^{(4)},\alpha^{(5)})=0$

\textbf{3.) }$\tau(s)=\pm\dfrac{s}{\left[  \tan^{2}\phi-s^{2}\right]  ^{1/2}}$

\begin{proof}
In the first step, it should be shown that the curve $\alpha$ is Salkowski if
and only if $\tau(s)=\pm\dfrac{s}{\left[  \tan^{2}\phi-s^{2}\right]  ^{1/2}}$
[3]. \ \ \ \ \ \ \ \ \ \ \ \ \ \ \ \ \ \ \ \ \ \ \ \ \ \ 

In the second step, we should prove that $\det(\alpha^{(3)},\alpha
^{(4)},\alpha^{(5)})=0$ if and only if
\[
\mathbf{\ }\tau(s)=\pm\dfrac{s}{\left[  \tan^{2}\phi-s^{2}\right]  ^{1/2}}.
\]
Here, by taking $c=0$ and $b=\dfrac{1}{\tan\phi}$ in Theorem 3, we get
\[
\tau(s)=\pm\dfrac{s}{\left[  \tan^{2}\phi-s^{2}\right]  ^{1/2}}.
\]
Hence, the theorem is proved.
\end{proof}
\end{theorem}
\end{proof}

\section{CONCLUSIONS}

The starting point of this study is to give the special space curves that are
characterized by $\det(\alpha^{(3)},\alpha^{(4)},\alpha^{(5)})=0,$ in
different approach. We have developed this approach with discussing main
properties of spherical curves, slant helices, Salkowski curves and
relationship between these curves. At this time, different approaches that we
give here have showed us the space curve is Salkowski if and only if
$\det(\alpha^{(3)},\alpha^{(4)},\alpha^{(5)})=0.$ Additionally, it is obtained
that the tangent indicatrice of the space curve $(T)$ is spherical helix if
and only if $\det(\alpha^{(3)},\alpha^{(4)},\alpha^{(5)})=0.$

\end{document}